\documentclass[11pt]{amsart}
\usepackage{amssymb, amsmath}
\usepackage{amsthm}
\usepackage{amscd}
\usepackage{graphicx}
\newtheorem{theorem}{Theorem}
\newtheorem{lemma}[theorem]{Lemma}

\newtheorem{question}[theorem]{Question}
\newtheorem{problem}[theorem]{Problem}

\newtheorem{conjecture}[theorem]{Conjecture}
\theoremstyle{definition}
\newtheorem*{remark}{Remark}
\newtheorem*{acknowledgement}{Acknowledgements}

\title{On finite volume, negatively curved manifolds}
\author{T. T$\hat{\mathrm{a}}$m Nguy$\tilde{\hat{\mathrm{e}}}$n Phan}
\address{Department of Mathematics\\
5734 S. University Ave.\\
Chicago, IL 60637}
\email{ttamnp@math.uchicago.edu}

\DeclareMathOperator{\MinVol}{MinVol}
\DeclareMathOperator{\Vol}{Vol}
\DeclareMathOperator{\Ker}{Ker}

\DeclareMathOperator{\vol}{Vol}

\DeclareMathOperator{\CAT}{CAT}

\def\R{\mathbb{R}}
\def\Z{\mathbb{Z}}
\def\N{\mathbb{N}}

\input xy
\xyoption{all}
\begin{document}
\begin{abstract}
We study noncompact, complete, finite volume, negatively curved manifolds $M$. We construct $M$ with infinitely generated fundamental groups in all dimensions $n \geq 2$. We construct $M$ whose cusp cross sections are compact hyperbolic manifolds in all dimensions $n\geq 3$. We construct nonuniform, negatively curved lattices that do not contain any parabolic isometries. We show that there are $M$, for each dimension $n \geq 3$, such that $\widetilde{M}$ does not satisfy the visibility axiom. We give a condition on the curvature growth versus the volume decay that guarantees topological finiteness. We raise a few questions on finite volume, negatively curved manifolds.
\end{abstract}
\maketitle

\section{Introduction}
Let $M$ be a noncompact, complete, finite volume manifold with sectional curvature $K(M)<0$. A fundamental theme in the study of nonpositively curved manifolds is the relation between curvature and topology. We concentrate on the following questions. 
\begin{itemize}
\item[1)] When is $M$ \emph{tame}, i.e. $M$ is the interior of some compact manifold $\overline{M}$ with boundary?
\item[2)] If $M$ is tame, which diffeomorphism types of manifolds can occur as $\partial \overline{M}$?
\item[3)] If $M$ is tame, what is the relationship between $K(M)$ and the topology of the inclusion map $i \colon \partial\overline{M} \longrightarrow \overline{M}$? For example, is the induced map  $i_*$ on the fundamental groups injective? Can $i_*$ be surjective?
\item[4)] Is each end of $M$ quasi-isometric to a ray? (See \cite{Eberleinlattices} for the definition of \emph{ends}.)
\end{itemize}
When $-1<K(M)<0$, some answers are known.
\begin{itemize}
\item If $-1< K(M) < 0$, then $M$ is tame by results of Gromov (\cite{Gromovneg}).

\item Eberlein proved that if $-1\leq K(M) \leq 0$ and the universal cover $\widetilde{M}$ is a visibility manifold (see the other end of the introduction), then $M$ is tame, and each end of $M$ is a \emph{parabolic end}, i.e. the quotient of a horoball of $\widetilde{M}$ by a group of parabolic isometries with compact cross section (\cite{Eberleinlattices}). This implies that each end of $M$ is quasi-isometric to a ray. Eberlein conjectured that if $M$ has no embedded half flats, then visibility of $\widetilde{M}$ follows (\cite[p. 438]{Eberleinlattices}). We will point out that this is not the case at the other end of the introduction.

\item Tameness is not guaranteed when $K(M) \rightarrow -\infty$. Eberlein constructed an example of an untame, surface with unbounded negative curvature (\cite{Eberleinlattices}).  
\end{itemize}

The goal of this paper is to exhibit new phenomena when the curvature conditions are relaxed, and to raise new questions about finite volume, negatively curved manifolds. 

\bigskip
\noindent
\textbf{Tameness.} We investigate which curvature conditions force tameness. Unless stated otherwise, all manifolds in this paper are connected, noncompact, complete and finite volume. We begin with untame examples. In contrast to Gromov's theorem on tameness, we have the following. 

\begin{theorem}[Examples with infinitely many ends]\label{inftyend}
For each dimension $n \geq 2$, there is a complete, finite volume, negatively curved $n$-manifold $M$ with infinitely many ends. In particular, $M$ is not the interior of a compact manifold with boundary. Moreover, one can give $M$ a complete, finite volume metric with sectional curvature $K(M) < -1$.
\end{theorem}

Eberlein constructed a two-ended surface with untame ends (\cite{Eberleinlattices}). We give a construction of such manifolds in all dimension $n > 2$. 
\begin{theorem}[Examples with untame ends in all dimension]\label{strictneg}
For each dimension $n > 2$, there is a complete, two-ended, finite volume, negatively curved manifold $M$ whose fundamental group is not finitely generated. There are such $M$ whose ends are not quasi-isometric to rays. Moreover, one can give $M$ complete, finite volume metrics with sectional curvature $K(M) < -1$.
\end{theorem}

\bigskip
\noindent
\textbf{Curvature growth and volume decay.} The examples that we give for Theorem \ref{strictneg} have the property that ``the growth of the curvature is about the same as the rate of decay of the volume". In trying to prevent such a  construction, we discovered the following phenomenon. 

Fix $p \in M$. For $r > 0$, we denote by $b_p(r)$ the supremum of the absolute value of the sectional curvature of all tangent planes at points in the ball $B_p(r)$. For each $r > 1$, let $A_p(r)$ be the annulus  $B_p\left(r\right) \setminus B_p\left(r - 1\right)$.
\begin{theorem}[Curvature growth - Volume decay]\label{CGVD}
Let $M$ be a complete, finite volume, negatively curved Riemannian manifold of dimension $n$. Suppose that 
\[\lim_{r\rightarrow \infty} b_p(r)^n\Vol(A_p(r))^2 = 0.\]
for some $p \in M$. Then $M$ is tame. 
\end{theorem}
We remark that for $n = 2$ there is an immediate proof of Theorem \ref{CGVD} using the Gauss-Bonnet theorem and negatively curved geometry. The limit condition in Theorem \ref{CGVD} is necessary. The following theorem is motivated by a question of Benson Farb: Is topological finiteness guaranteed if the growth of the curvature is slow enough?  
\begin{theorem}\label{growth}
Let $f \colon (1,\infty) \longrightarrow \R^+$ be any non-decreasing function. Then there exists a complete, finite volume, negatively curved manifold $M$ with infinitely generated fundamental group with a point $p \in M$ such that $b_p(r) < f(r)$ for all $r \in (1, \infty )$.
\end{theorem}

\bigskip
\noindent
\textbf{Cusp cross sections of tame manifolds.} Now we discuss the class of tame, negatively curved manifolds. Let $M^n$ be the interior of a compact, smooth manifold $\overline{M}$ with boundary. Suppose that $M$ has a negatively curved Riemannian metric with finite volume. We can remove a compact subset $A$ of $M$ such that  $(M-A)$ is diffeomorphic to $\partial\overline{M} \times (0,\infty)$. By compactness of $M$, the set $(M-A)$ has finitely many components. Each component of $(M-A)$ is diffeomorphic to $C \times (0,\infty)$, for some compact manifold $C^{n-1}$ diffeomorphic to a boundary component of $\overline{M}$. We call such connected components of $(M-A)$ \emph{cusps} of $M$. For each cusp $C\times (0, \infty)$, we say that $C$ is the \emph{cusp cross section}.

\begin{problem}
Classify all closed $(n-1)$-manifolds that can occur as cusp cross sections of a tame, finite volume, negatively curved manifold.
\end{problem} 

If $K(M)$ is \emph{pinched}, i.e. $-1 < K(M) < -a^2 $ for some $0 < a <1$, then the cusp cross sections of $M$ are compact infra-nilmanifolds (\cite{BuserKarcher}, \cite{Eberleinlattices}). So, in some sense, they are small relative to $M$. There are manifolds $M$ with $-1 <K(M) < 0$ and $K(M) \rightarrow 0$ that do not admit a pinched negatively curved metric. Fujiwara (\cite{Fujiwara}) constructed manifolds $M$ with $-1\leq K(M) <0$ with cusp cross sections that are circle bundles over a compact hyperbolic manifold. Belegradek (\cite{Belegradekch}) constructed manifolds $M$ with $K(M)<-1$  whose cusp cross section are circle bundles over a compact complex hyperbolic manifold. Abresch and Schroeder (\cite{AS}) constructed negatively curved manifolds $M$ with $-1\leq K(M)<0$with cusp cross section a generalized graph manifold. 


If $-1<K(M)<0$, then the cusp cross sections have zero Euler characteristic by a theorem of Cheeger and Gromov (\cite[Theorem 1.2]{CGL2}). They also have zero simplicial volume by results of Gromov (\cite[p. 17]{Gromovbounded}). We will also explain this in Section \ref{sec:0simpvol}. 

In this sense, the cusp cross sections when $-1 <K(M) < 0$ are small compared to the manifold. One can ask whether the same thing holds if we only restrict to negative curvature $K(M) <0$ and impose no other bounds on the curvature. For example, can the cusp cross section have non-zero simplicial volume or Euler characteristic? Can the fundamental group of a cusp cross section be $\delta$-hyperbolic? We give a positive answer to these questions, and thus give a sharp contrast with the restriction on the simplicial volume and Euler characteristic of the cusp cross section when $-1 <K(M) < 0$. For a manifold $N$, we denote by $||N||$ the simplicial volume of $N$ (also called the ``Gromov norm" of $N$). 

\begin{theorem}[Examples of cusps with non-zero simplicial volume]\label{hyperbolic cusp}
For each dimension $n \geq 3$, there are complete, finite volume, negatively curved manifolds $M^n$ whose cusp cross sections are compact hyperbolic manifolds $C^{n-1}$. In particular, $||C|| > 0$ for all $n \geq 3$, and $\chi (C) \ne 0$ for odd $n$.
\end{theorem}

We see that Theorem \ref{hyperbolic cusp}, together with results of Fujiwara (\cite{Fujiwara}) and Abresch and Schroeder (\cite{AS}) imply the following theorem. 
\begin{theorem}\label{differentclasses}
The following three classes of noncompact, complete, finite volume, negatively curved manifolds are distinct. 
\begin{itemize}
\item[1.] $-1< K < -a^2$ for some $0 <a <1$,
\item[2.] $-1<K<0$, 
\item[3.] $K <-1$, 
\end{itemize}   
\end{theorem}

\begin{question}\label{question: K->0}
Is there a noncompact, finite volume, negatively curved manifold $M$ that does not admit any metric with $K(M) < -1$? That is, any negatively curved metric on $M$ must have $K(M) \rightarrow 0$, and $K(M) \rightarrow -\infty$.  More generally, what can guarantee $K(M) \rightarrow 0$?
\end{question}


\bigskip
\noindent
\textbf{Injectivity and surjectivity of $\pi_1(C) \longrightarrow \pi_1(M)$.}  
Let $M$ be a tame, negatively curved manifold, and let $C$ be a cusp cross section of $M$. If the curvature of $M$ is pinched, i.e. $-1< K(M) <-a^2 <0$, then the structure of $M$ is similar to that of real rank one locally symmetric manifolds (e.g. hyperbolic manifolds) in that:
\begin{itemize}
\item[a)] Each end of $M$ is the quotient of a horoball in $\widetilde{M}$ by a group of parabolic isometries fixing the same point at infinity.  
\item[b)] $\pi_1(C)$ is a proper subgroup of $\pi_1(M)$.
\item[c)] There is a loop $\gamma$ in $M$ that is \emph{stuck}, i.e. there is some compact set $A$ such that $\gamma$ cannot be homotoped to lie in the complement of $A$.
\item[d)] $M$ has finitely many ends.
\end{itemize}

On the other hand, in the class of finite volume, nonpositively curved manifolds with bounded curvature, there are examples for which none of the above properties hold. For examples, the product of two noncompact surfaces does not satisfy the first three of above properties. There are finite volume, nonpositively curved manifolds $M$ with bounded curvature that violate  property (d), i.e. $M$ has infinitely many ends. We will give an example of such manifolds in Section \ref{sec: infinite ends}. 

Therefore it is interesting to investigate which of the above properties of tame, negatively curved manifolds remain when one relaxes the curvature conditions, i.e. whether the above topological properties come from negativity of the curvature or the pinching of the curvature. As in our discussion on topological finiteness, there are negatively curved manifolds that do not satisfy (d). Buyalo contructed $4$-dimensional manifolds $M$ with $-1 <K(M) <0$ for which $i_*\colon\pi_1(C) \longrightarrow \pi_1(M)$ is not injective. So these manifolds are counterexamples to (b). More surprisingly, if we do not require the curvature to be bounded, the first three properties fail for the following examples. 

\begin{theorem}[Examples of manifolds with peripheral $\pi_1$]\label{peripheral}
Let $\Sigma_k$ be a compact surface of genus $k \geq 2$. There is a complete, finite volume, negatively curved Riemannian metric $g$ on the manifold $X = \Sigma_k\times\R$.
\end{theorem}

Theorem \ref{peripheral} is based on an idea of Ursula Hamenst\"adt communicated to me by Grigori Avramidi, who asked a related question: does a handlebody admit a complete, finite volume, negatively curved metric? Hamenst\"adt suggested the infinite cyclic cover of a noncompact $3$-manifold fibering over a circle (which is topologically a handlebody) might have such a metric. I have tried this unsuccessfully, and can only make it work for the case of the cyclic cover of a compact hyperbolic manifold fibering over the circle.   

One might expect that for any noncompact, finite volume, negatively curved manifold $M$, the group $\pi_1(M)$ of deck transformations on the universal cover $\widetilde{M}$ should contain some parabolic isometries, as in the case of locally symmetric spaces. However, a surprising property of the following examples is that the entire fundamental group $\pi_1(M)$ acts by semisimple isometries. To the best of our knowledge, this is the first example of a noncompact, tame, complete, finite volume manifold $M$ with $K(M) <0$ that has this property. 

\begin{theorem}[Examples of $\pi_1(M)$ with no parabolic isometries] \label{semisimplelattice}
Let $\Sigma_k$ be a compact surface of genus $k \geq 2$. There is a Riemannian metric $g$ on the manifold $X = \Sigma_k\times\R$ that is complete, negatively curved, the volume $\Vol(X,g)$ is finite, and $\pi_1(X)$ contains only semisimple isometries of $\widetilde{X}$.  
\end{theorem}
We summarize the discussion so far in the following table. 

\bigskip
\begin{center}
\begin{tabular}{|r|r|r|r|r|}
\hline
\textbf{Curvature} & \textbf{Cusp cross section}  & \textbf{Inj. of $\pi_1(C)$} & \textbf{No stuck loops} & \textbf{Para. ends}  \\
\hline
$-1<K<-a^2$ 	& is a compact & True (\cite{Eberleinlattices})   	& True	     	& True \\
				& infra-nil manifold.	&   & & \\

\hline
$-1< K <0$ 		& $\chi (C) = 0$,  $||C|| = 0$  
	& False (\cite{Buyalo})		& ? 			& ?\\
\hline
$K<-1$ 			& can be a compact & ? 		& False 			& False \\
				& hyperbolic manifold, &  	&(see Theorem \ref{peripheral}) 	& \\
				& $||C||$ can be non-zero & & & \\
\hline
\end{tabular}
\newline
\end{center}  
\bigskip
\noindent

There is a related conjecture of Farb about nonpositively curved manifolds with geometric rank one. A nonpositively curved manifold has \emph{geometric rank one} if there is a geodesic in its universal cover that does not bound an infinitesimal half flat.

\begin{conjecture}[Farb]
Let $M$ be a tame, complete, finite volume manifold that has geometric rank one. Suppose that $-1 \leq K(M) \leq 0$. Then there is a nontrivial loop in $M$ that cannot be homotoped to leave every compact set.  
\end{conjecture}

\bigskip
\noindent
\textbf{Visibility.} A simply connected, complete, nonpositively curved manifold $X$ is a \emph{visibility} manifold (or is said to satisfy the visibility axiom) if any pair of distinct points of the boundary at infinity of $\partial_\infty X$ can be connected by a geodesic. If $X$ has strictly negative curvature, i.e. $K(X) < -a^2 <0$, then $X$ is a visibility manifold (\cite{Eberlein}). In general, a simply connected, nonpositively curved manifold can not be a visibility manifold if it contains half flats. If $K(M) <0$, then $\widetilde{M}$ does not contain any flat strips, let alone half flats. However, negative curvature does not imply visibility. Though this is surely known to the experts, we could not find a reference with such a concrete counterexample. We give such a concrete example in Section \ref{sec: invisibility}. 

If $M$ is a compact negatively curved manifold, then the universal cover $\widetilde{M}$ is a visibility manifold since $K(M)$ is pinched. Eberlein conjectured (\cite{Eberleinlattices}) that if $M$ is a noncompact, finite volume, manifold with $-1\leq K \leq 0$, and $\widetilde{M}$ contains no half flats, then $\widetilde{M}$ is a visibility manifold. He also proved that this is true if $M$ is a non-flat surface (\cite[Proposition 2.5]{Eberleinsurface}). While it is true that if $M$ has finite volume and $K(M) <0$ (which implies that $\widetilde{M}$ does not contain any half flats), then the set of end points of geodesics is dense in $\partial_\infty \widetilde{M}\times \partial_\infty \widetilde{M}$ (see \cite[p. 45, Theorem 3.4]{Ballmann}), we show that there are counterexamples to Eberlein's conjecture for all dimensions $n>2$.  
\begin{theorem}\label{nonvisibility lattice}
Let $M$ be a finite volume manifold with $-1 \leq K(M) <0$ that is obtained by the Fujiwara construction in \cite{Fujiwara}. Assume that $M$ has dimension $n >2$. Then the universal cover $\widetilde{M}$ of $M$ is not a visibility manifold.   
\end{theorem}
The $4$-manifolds $M$ constructed by Buyalo in \cite{Buyalo} has curvature $-1\leq K<0$ and the map $i_* \colon \pi_1(C) \longrightarrow \pi_1(M)$ is not injective. This implies, by results of Eberlein (\cite{Eberleinlattices}), that the universal cover $\widetilde{M}$ is not a visibility manifold. 

The proof of Theorem \ref{nonvisibility lattice} relies on the existence of a totally geodesic surface in $M$ that is isometric to a surface of revolution $S$ in $\R^3$ whose width has a positive lower bound, and the conservation of angular momentum of geodesics on $S$. The positive lower bound on the width of $S$ comes from the fact that there is a parabolic deck transformation $\gamma\in \pi_1(M)$ such that the displacement function $d(x,\gamma (x))$ has positive infimum. This tempts us to state the following (possibly naive) conjecture and question.
\begin{conjecture}
Let $M$ be a tame, finite volume, negatively curved manifold. Suppose that there is a parabolic element $\gamma \in \pi_1(M)$ such that
\[\inf_{x\in\widetilde{M}} d(x,\gamma (x)) >0.\]
Then $\widetilde{M}$ does not satisfy the visibility axiom.
\end{conjecture}

\begin{question}
Let $M$ be a tame, finite volume, negatively curved manifold. Suppose that for each parabolic element $\gamma \in \pi_1(M)$, we have
\[\inf_{x\in\widetilde{M}} d(x,\gamma (x)) = 0.\]
Then does $\widetilde{M}$ need to satisfy the visibility axiom?
\end{question}
This paper is organized as follows. We prove Theorem \ref{inftyend} in Section \ref{sec: infinite ends}. We prove Theorems \ref{strictneg} and \ref{growth} in Section \ref{nonfinitetop}. We prove Theorem \ref{CGVD} in Section \ref{sec:CGVD}.  We prove Theorem \ref{0simpvol} in Section \ref{sec:0simpvol}. We prove Theorem \ref{hyperbolic cusp} in Section \ref{hypcusp}. We prove Theorem \ref{differentclasses} in Section \ref{sec:different classes}. We prove Theorems \ref{peripheral} and \ref{semisimplelattice} in Section \ref{sec: peripheral}. We prove Theorem \ref{nonvisibility lattice} in Section \ref{sec: invisibility}.

\begin{acknowledgement} I would like to thank my advisor, Benson Farb, for inspiring conversations, for opening me to interesting questions on negatively curved manifolds, and for giving extensive comments on earlier versions of this paper. I would like to thank Grigori Avramidi for useful discussions and for suggesting to me various good ideas and questions. I would like to thank Ilya Gekhtman and Shmuel Weinberger for useful conversations. I would like to thank Aaron Marcus for help with the curvature calculations. I would like you thank Igor Belegradek for useful mathematical comments about the paper, and for pointing me to interesting papers on this topics. I would like to thank Grigori Avramidi and Thomas Church for commenting on earlier versions of this paper. 
\end{acknowledgement}

\section{Cusp cross sections have zero simplicial volume if $-1 < K <0$}\label{sec:0simpvol}
\begin{theorem}[Small cusps]\label{0simpvol}
Let $M$ be a noncompact, complete, finite volume manifold with $-1<K(M)<0$. Suppose that $M$ is diffeomorphic to the interior of a compact manifold with boundary components $N_i$, for $i = 1, 2, ..., k$. Then for each $i$, we have $||N_i|| = 0$.
\end{theorem}

The main ingredient of the above theorem is the following result of Cheeger and Gromov (\cite{CG1}). 
\begin{theorem}[Cheeger-Gromov]
Let $M^n$ be a complete, finite volume manifold and the sectional curvature $K$ of $M$ is bounded $|K| < 1$. Then $M$ admits an exhaustion $M = \cup_i^\infty M^n_i$ by manifolds $M_i$ with boundary such that the volume $\Vol(\partial M_i)\rightarrow 0$ and the second fundamental form of the metric restricted to $\partial M_i$ are uniformly bounded by a constant $c(n)$. 
\end{theorem}
\begin{proof}[Proof of Theorem \ref{0simpvol}]
As pointed out to me by Grigori Avramidi, for large enough $i$, there exists a degree $1$ map $f \colon \partial M_i \longrightarrow C$, where $C$ is a cusp cross section. Hence,
\[||C|| \leq ||\partial M_i||.\]
By the above theorem, the MinVol of $\partial M_i$ tends to $0$ as $i \rightarrow \infty$. Hence, 
\[||\partial M_i || \leq \text{const}(n)\cdot\MinVol (\partial M_i) \rightarrow 0.\]
Therefore, $||C|| = 0$. 
\end{proof}
\begin{remark}
The same conclusion of Theorem \ref{0simpvol} holds if the curvature is bounded with no restriction of the sign of curvature, i.e. $-1< K < 1$. 
\end{remark}
\section{Proof of Theorem \ref{hyperbolic cusp}}\label{hypcusp}
\subsection{The construction}
Let $M$ be a compact hyperbolic manifold of dimension $n$. Let $N$ be a totally geodesic, embedded, codimension $1$ submanifold of $M$. In this case, $N$ is a compact hyperbolic manifold of dimension $(n-1)$. The hyperbolic metric of $M$ on an $\delta$-neighborhood $T$ of $N$ is 
\[g_{hyp}^M = dt^2 + \cosh^2(t)g_{hyp}^N, \]
where $h_N$ is the hyperbolic metric on $N$ and $t$ is the distance from a point in $T$ to $N$. 

Now, for $\epsilon < \delta/2$, we delete the $\epsilon$-tubular neighborhood $W$ of $N$ corresponding to $r \in (-\epsilon, \epsilon)$. If $N$ is separating, then we obtain two manifolds with boundary. Otherwise, the resulting space is a manifold with two boundary components. In either case, we will modify the metric of the interior manifold near the boundary to get a complete, negatively curved manifold with no boundary.   

Let $X$ be a component of $M\setminus W$ and let $Y$ be a component of the boundary of $X$. Let $Y_\epsilon$ be the $\epsilon$-neighborhood (with respect to the hyperbolic metric on $X$) of $Y$. Before, doing any modification to the hyperbolic metric, we need to reparametrize the open set $Y_\epsilon \setminus Y$. The $x_i$'s coordinates remains the same, we just stretch the $t$ coordinates as follow. Let 
\[\alpha \colon (\epsilon, 2\epsilon) \longrightarrow (-2\epsilon, \infty)\]
such that the interval $(3\epsilon/2, 2\epsilon)$ is map isometrically (with respect to the standard metric on $\R$) to $(-2\epsilon,-3\epsilon/2)$, and the interval $(\epsilon, 3\epsilon/2)$ is mapped to $(-3\epsilon/2, \infty)$. We are going to abuse notation and call this reparametrization $t$. 

The metric we will give $Y_\epsilon \setminus Y$ is the following 
\[ g = dt^2 + \dfrac{4f^2(t)}{(1-r^2)^2}(dx_1^2 +dx_2^2 + ... + dx_{n-1}^2),\]
for $r^2 = x_1^2 + ... + x_{n-1}^2$ and some function $f$ that satisfies the conditions in the following lemma whose proof we leave as an exercise for the reader.
\begin{lemma}
Let $a < 0$. There is a convex function $f \colon [a, \infty) \longrightarrow (0, \infty)$ that agrees with $\cosh(t)$ at $t= a$ for at least the value of the function, the first and second derivatives (from the right) at $a$, and
\[\lim_{t \rightarrow \infty} t^3f(t) = 0.\] 
\end{lemma}

The number $a$ in the above lemma will be chosen to be $-3\epsilon/2$. By the curvature calculation in the next section, if we pick $f$ to be a function that satisfies the condition in the lemma, then the manifold $Z = X \setminus Y$ is $C^2$ negatively curved with respect to the metric $h$ obtained by gluing the hyperbolic metric on $X \setminus Y_\epsilon$ to the metric $g$ on $Y_\epsilon \setminus Y$, and the curvature is unbounded from below. The manifold $(Z, h)$ is complete since 
\[\int_a^\infty 1 dt = \infty,\]  
and has finite volume since
\[\vol(Z,h) = \vol(X\setminus Y_\epsilon, g_{hyp}) + \vol(Y_\epsilon, g).\]
Clearly the first term on the right hand side is finite. The second term
\[\vol(Y_\epsilon, g) = \int_a^\infty \vol(N, g_{hyp}^N)f(t) dt,\]
which is finite since $f(t)$ decays faster than $1/t^3$.

\begin{remark} 
One can replace $t^3$ by $e^t$ in the above theorem and make $K <-1$.
\end{remark}

\begin{question} 
Do these manifolds admit an analytic metric with negative sectional curvature?
\end{question}

\subsection{Curvature calculations}
Let $g$ be the following metric.
\[ g = dt^2 + \dfrac{4f^2(t)}{(1-r^2)^2}(dx_1^2 +dx_2^2 + ... + dx_{n-1}^2),\]
for $r = \sqrt{x_1^2 +x_2^2 + ... + x_{n-1}^2}$ and $f(t)$ is a positive, convex function defined for $t \in (0,\infty)$. Then the nonzero components of the curvature tensor are
\[R(\dfrac{\partial}{\partial t}, \dfrac{\partial}{\partial x_i}, \dfrac{\partial}{\partial t}, \dfrac{\partial}{\partial x_i}) = -\dfrac{4f(t)f''(t)}{(1-r^2)^2},\]
\[R(\dfrac{\partial }{\partial x_i}, \dfrac{\partial}{\partial x_j}, \dfrac{\partial}{\partial x_i}, \dfrac{\partial}{\partial x_j}) = \dfrac{16f(t)^2(1+f'(t)^2)}{(1-r^2)^4}.\]
Hence, the sectional curvature at a point of the corresponding $3$ planes are negative, as below
\[K(\dfrac{\partial }{\partial t}, \dfrac{\partial }{\partial x_i}) = -\dfrac{f''(t)}{f(t)},\]
\[K(\dfrac{\partial }{\partial x}, \dfrac{\partial }{\partial y}) = -\dfrac{1+f'(t)^2}{f(t)^2}.\]
(Observe that if $f(t) = \cosh^2(t)$, then $g$ is the hyperbolic metric). 

Since all cross terms of the curvature tensor are $0$, the sectional curvature of a plane is a convex combination of the above numbers, and thus, is negative. Observe that for $i \ne j$,
\[ \lim_{t\rightarrow \infty}K(\dfrac{\partial }{\partial x_i}, \dfrac{\partial }{\partial x_j}) =  \lim_{t\rightarrow \infty} -\dfrac{1+f'(t)^2}{f(t)^2} = -\infty\] 
since $f(t) \rightarrow 0$ as $t \rightarrow \infty$. So $M$ is a negatively curved manifold whose sectional curvature is unbounded from below. We would like remark that if one pick $f$ to be eventually a multiple of $e^{-t}$, we can make the curvature of $M$ bounded away from $0$.

\section{Different classes of tame, negatively curved manifolds}\label{sec:different classes}
We discuss Proposition \ref{differentclasses} in this section. First, we start with the case in which the curvature is pinched, i.e. $-a^2 < K < -1$. This includes locally symmetric spaces of rank $1$. The fundametal group of a cusp is virtually abelian if $M$ is hyperbolic, and is virtually $2$-step nilpotent otherwise. Benson Farb told me that it is not known whether there is a pinched negatively curved manifold whose cusp is nilpotent with more than $2$ steps.

The examples in \cite{Fujiwara} and \cite{AS} are of dimension at least $4$ and are obtained by firstly removing codimension $2$, totally geodesic submanifold(s) $N$ of a compact hyperbolic manifold $M$, and then stretching out the hyperbolic metric around $N$ to make the resulting space $X = M - N$ complete without changing the sign of the curvature. Since the cusp cross section of $X$ will be a circle (unit normal) bundle over the compact hyperbolic manifold $N$, it is not an infra-nilmanifold. Therefore, $X$ does not admit pinched negatively curved metric.  

The examples in the proof of Theorem \ref{hyperbolic cusp} has curvature bounded away from $0$ and approaching $-\infty$. The cusp cross section $C$ is a compact hyperbolic manifold, which has non-zero simplicial volume (by Theorem \ref{0simpvol}). So the these manifolds do not belong to the above two classes. 

Combining the techniques in \cite{Fujiwara}, \cite{AS} and that in the proof of Theorem \ref{hyperbolic cusp}, we can construct a negatively curved manifold $M$ of dimension at least $4$ with two ends, one of which has curvature $K \rightarrow0$ and the other one has $K \rightarrow -1$. It is not hard to see that any negatively curved metric on $M$ must have curvature $K \rightarrow -\infty$. It will be interesting to prove (or disprove) that any negatively curved metric on $M$ must also have $K \rightarrow 0$. (See Question \ref{question: K->0} in the Introduction.)

\section{Examples of negatively curved manifolds with peripheral fundamental groups}\label{sec: peripheral}
We will prove Theorem \ref{peripheral} and Theorem \ref{semisimplelattice} at the end of this section. Firstly we need to prove the following theorem, which also gives another way to obtain finite volume, negatively curved manifolds (that may or may not be of finite topological type).
 
\begin{theorem}\label{cycliccover}
Let $M$ be a compact negatively curved manifold. Suppose that there is a surjection $\phi \colon \pi_1(M) \longrightarrow\Z$.  Let $X$ be the cover of $M$ that corresponds to $\Ker\phi$. Then there is a complete, finite volume, negatively curved Riemannian metric on $X$. 
\end{theorem}
  
The above theorem was first motivated by the example in Theorem \ref{peripheral}. An example of such an $M$ is a surface bundle over a circle with pseudo-Anosov monodromy. Thurston proved that such manifolds have a hyperbolic metric. The generality of the above theorem was suggested to me by Shmuel Weinberger.

\begin{proof}
There is a smooth map $f \colon M \rightarrow \mathbb{S}^1$. Let $a$ be a regular value of $f$, and let $S = f^{-1}(a)$. Such an $a$ exists by Sard's theorem. Then $S$ is a submanifold of $M$. Let $I$ be an interval containing $a$ that contains only regular values of $f$. Let $U = f^{-1}(I) \cong S\times I$. We assume that $I = [-\frac{1}{2},\frac{1}{2}]$ by a change of coordinates, and suppose that $a = 0$. Choose a local parametrization of $U$ so that $U = S\times I$. We don't need to do all of this but this is good psychologically for the author (and hopefully for the readers).

Pick a basepoint $s_0 \in S$ and pick a lift $x_0$ of $s_0$ on $X$. Let $V_0$ be the lift of $U$ that contains $x_0$. The group $\Z$ acts on $X$ via covering space transformations. For each $k \in \Z$, let $V_k$ be the image of $V_0$ under transformation corresponding to $k$.

Let $g$ be the metric on $X$ that is lifted from the negatively curved metric on $M$. We perturb the metric $g$ on $X$ within $V_0$ by scaling $g$ by a function $c_0(t)$. We can pick $c_0 \colon I \longrightarrow \R$ to be a smooth function such that $c_0([-\frac{1}{2},-\frac{1}{4}]) = \{1\}$, and $c_0([\frac{1}{4},\frac{1}{2}]) = \{1-\epsilon\}$ for some $\epsilon$, and the absolute value of the first and second derivatives of $c_0$ is smaller than $\varepsilon$ for some $\varepsilon$ so that the sectional curvature $K_{c_0.g}$ of $(V_0, c_0.g)$ is negative. This can be done since $K_g$ is negative and bounded away from $0$. 

Let $S_0$ be the lift of $S$ that contains $x_0$. For each $k \in \Z$, let $S_k$ be the image of $S_0$ under the covering space transformation corresponding to $k$. Let $D_k$ be a fundamental domain of $X$ that is bounded by $S_k$ and $S_{k+1}$. The set $V_0$ is diffeomorphic to $S_0\times[-\frac{1}{2},\frac{1}{2}]$. Let $S_0^+ = S_0\times\frac{1}{2}$ and let $S_0^- = S_0\times-\frac{1}{2}$. We assume that $S_0^+$ is contained in the interior of $D_0$. We define $S_k^+$ and $S_k^-$ the obvious way. 

For $D_0$, we scale the metric $g$ on $(D_0 - (V_0\cup V_1))$ by $c_0(\frac{1}{2}) = 1 - \epsilon$. Call the new metric $\hat{g}$. Since $c_0$ is constant on a connected neighborhood of $\frac{1}{2}$, the rescaled metric $\hat{g}$ on $(D_0 - V_1)$ is smooth. It is not hard to see that the sectional curvature of $(D_0, \hat{g})$ is negative.     

Let $m$ be a positive integer such that $\frac{m-1}{m} < (1- \epsilon)$. We could have picked $\epsilon$ above to be such that $(1 - \epsilon)^d = \frac{1}{m}$ for some $d \in \N^+$. For each $k < d$, let $c_k \colon I \longrightarrow \R$ be defined the same as $c_0$ but with $(1-\epsilon)$ replaced by $(1-\epsilon)^{k+1}$, and $c_k([-\frac{1}{2},-\frac{1}{4}]) = \{(1-\epsilon)^k\}$. For $k \geq d$, let $c_k \colon I \longrightarrow \R$ be defined the same as $c_0$ but with $(1-\epsilon)$ replaced by $\frac{1}{k-d+m+1}$, and $c_k([-\frac{1}{2},-\frac{1}{4}]) = \{\frac{1}{k-d+m}.\}$.

We rescale the metric $g$ on $V_k$'s and $D_k$'s using the function $c_k$ in the same way as we did for $D_0$. This is how we extend the metric $\hat{g}$ smoothly to $X$. It is not hard to see that the sectional curvature $K_{\hat{g}}$ is negative at every point in $X$.

We are left to show that $(X,\hat{g})$ is complete and $\Vol(X,\hat{g}) < \infty$. To show that these are true reduces to showing that
\[ \sum_{k\in\N} \dfrac{1}{k} = \infty,\]
which implies that $(X, \hat{g})$ is complete; and
\[\sum_{k\in\N} \left(\dfrac{1}{k}\right)^n < \infty\]
for $n > 1$, which implies that $\Vol(X,\hat{g})$ is finite. But the convergence properties of the above two infinite sums follow from the $p$-test.

We modify the metric of $X$ on the other end in the same way to get the desired metric. 
\end{proof} 

\begin{proof}[Proof of Theorem \ref{peripheral}] 
Let $M$ be a bundle over the circle $\mathbb{S}^1$ with fiber a compact surface $\Sigma_k$ and pseudo-Anosov monodromy. Then there is a sujection $\pi_1(M) \longrightarrow\mathbb{S}^1 \cong \Z$. By Theorem \ref{cycliccover}, the manifold $M$ has a complete, finite volume, negatively curved Riemannian metric. 
\end{proof}

\begin{proof}[Proof of Theorem \ref{semisimplelattice}]
By construction, the metric constructed on $X$ in the proof of Theorem \ref{peripheral} has curvature $K < -1$ up to scaling. It follows that if $\gamma$ is an isometry of $\widetilde{X}$, then
\[\inf_{x\in \widetilde{X}}d_{\hat{g}}(x,\gamma(x)) = 0.\]

Suppose that there is a parabolic deck transformation $\gamma \in\pi_1(X)$. Then there is a sequence $x_i$, for $i = 1,2, 3,...$ such that $d_{\hat{g}}(x_i,\gamma (x_i)) \rightarrow 0$. 
Let $a_{\gamma}$ be the axis of translation of $\gamma$. With respect to the hyperbolic metric on $X$, the displacement $d_{g_{hyp}}(x,\gamma(x))$ grows exponentially with respect to the distance $\rho(x):= d_{g_{hyp}}(x,a_\gamma)$. 

By construction of the metric $\hat{g}$, we have $d_{\hat{g}}(x_i,\gamma (x_i))$ grows at least $e^{\rho(x)}/\rho(x)$, which tends to $\infty$ as $d_{g_{hyp}}(x,a_\gamma)\rightarrow \infty$. This is a contradiction to the assumption that $\gamma$ is a parabolic deck transformation. Hence, with respect to the metric $\hat{g}$, all non-identity elements of  $\pi_1(X)$ are hyperbolic isometries. 
\end{proof}

\section{Negatively curved manifolds with untame ends}\label{nonfinitetop}
\subsection{The general theme.}

As in Section \ref{hypcusp}, let $Y^n$ be a compact hyperbolic manifold of dimension $n \geq 2$. Let $N^{n-1}$ be a union of disjoint, compact, totally geodesic submanifolds of $Y$. Let $X$ be the compact manifold with boundary whose interior is diffeomorphic to a connected component of $Y\setminus N$. We call $X$ a \emph{block}.

The general theme of the constructions that appear in this section and the next section is the following. 
\begin{itemize}

\item[1)] Pick a graph $G$ with a based vertex $v_0$. We take blocks with the right number of boundary components and glue them along pairs of boundary components to obtain a noncompact manifold $M$ that has the structure of a \emph{graph of spaces} with underlying graph $G$. The vertex spaces are the pieces, and edge spaces are boundary components of the pieces. Let $V(G)$ be the set of vertices of $G$, and let $E(G)$ be the set of edges of $G$. Assume that each edge of $G$ has length $1$. For each vertex $v \in G$, let $w(v)$ be the distance in $G$ from $v$ to $v_0$.

\item[2)] Let $X_0$ be the block corresponding to $v_0$. We give $X_0$, a truncated negatively curved metric with $K < -1$, where the truncation is to be chosen later. Let  
\[\lambda \colon \N \longrightarrow (0,\infty) \]
be an increasing function to be chosen later. For each vertex $v$, we give the block corresponding to $v$ a truncated negatively curved metric with $K < -1$ but scaled by a factor of $1/{\lambda(w(v))}$, where the truncation is to be chosen later.

\item[3)] We pick the truncation for the metric of each block, starting from $X_0$ and move radially out, such that the gluing on each pair of boundary components is via an isometry. Then we smooth the metrics at the gluing while keeping the volume finite but keeping the curvature bounded from above by $-1/10^{10}$, and making sure that the metric is complete. Let $h$ be the obtained metric.
\end{itemize}

Then $(M,h)$ will be the desired manifold.

\bigskip
\noindent
\textbf{Gluing and smoothing out metrics.} We elaborate on steps $2$ and $3$. We give $Y\setminus N$ the complete, finite volume, negatively curved metric $g_0$ constructed in the proof of Theorem \ref{hyperbolic cusp}. For each cusp $C$ of $Y\setminus N$ the metric $g_0$ has the form 

\[g_0 = dt^2 + \dfrac{4f^2(t)}{(1-r^2)^2}(dx_1^2 +dx_2^2 + ... + dx_{n-1}^2),\]
for $t \in (0,\infty)$ and $x_i$ are coordinates of a cross section $S$ of the cusp.
For a block $X$, we give $X$ be the metric $g_X = \dfrac{1}{\lambda(w(v))^2}\cdot g_0$ but truncated at $t = T$, for some large $T$ to be picked later. 

If we multiply $g_0$ by a constant $1/A^2 = 1/\lambda(w(v))^2$, then we get
\[\dfrac{1}{A^2}g_0 = \dfrac{1}{A^2}\cdot\left(dt^2 + \dfrac{4f^2(t)}{(1-r^2)^2}(dx_1^2 +dx_2^2 + ... + dx_{n-1}^2)\right),\]
which can be written as, under the change of coordinates $\tau = t/A + \text{const}$ and $x_i = x_i$,
\[\dfrac{1}{A^2}g_0 = d\tau^2 + \dfrac{4f^2(A\tau)}{A^2(1-r^2)^2}(dx_1^2 +dx_2^2 + ... + dx_{n-1}^2),\]
We pick $f(t) = e^{-t}$ for large enough $t$. Then $f(At) = e^{-At}$.

Now we address the truncation of $X$. Suppose that two blocks $(X,g_X)$ and and $(X',g_{X'})$ are glued along their boundary components $C$ and $C'$ respectively. In all the constructions that we will carry out the cross section $C$ and $C'$ with the restricted metrics $g_X$ and $g_{X'}$ respectively are homothetic. Let $g_C$ and $g_{C'}$ be the metric on the boundary $C$ and $C'$ respectively. We choose the truncation $t = T$ large enough so that $g_C = g_{C'}$, and that the diameter of $(X,g_X)$ is at least $1$ (this is to guarantee completeness of $M$ later). 

We pick the function $\lambda\colon\N\longrightarrow (0,\infty)$ mentioned above based on the graph $G$. The point is that we pick $\lambda$ such that the growth of $\lambda$ is fast enough to guarantee finiteness of the infinite sum of the volume of all the blocks. One can always do this if $G$ is locally finite. 

Let $g$ be the Riemannian metric $M$ obtained by gluing the metrics on the blocks of $M$ together. Then $g$ is continuous but not $C^1$, and $g$ is complete with sectional curvature (whenever it is defined) $K <-1$. The non-smooth part of the metric occurs at the gluing. (We remark the path metric induced by $g$ is locally $\CAT(-1)$.)  

Now we need to smooth out the metric $g$ to get a smooth metric $h$ without altering completeness, finite volume while keeping the curvature $K > -10^{-10}$. This reduces to Lemma \ref{smoothing}.
\begin{lemma}\label{smoothing}
Let $f \colon \R \longrightarrow \R$ be defined as
\[f(t) = 
\begin{cases}
&  e^{-A(t+2a)}  \qquad \text{if} \quad t \leq 0 \\
&  e^{2A(t-a)} \:\;  \qquad \text{if} \quad t \geq 0.
\end{cases},\]
for some constant $A \geq 2$. Then there exists a $C^2$ function $h \colon \R \longrightarrow \R$ that agrees with $f$ outside $(-1/A,1/A)$ such that $\dfrac{h''(t)}{h(t)} > 10^{-10}$ for all $t \in \R$.
\end{lemma} 

\begin{proof}
Calculus exercise for the readers.
\end{proof}

By rescaling the metric by a constant, we make the curvature $K < -1$ (instead of $K < -10^{-10}$). Let $h$ be the rescaled smoothed out metric. Then $(M,h)$ is the desired manifold. 

Now we prove Theorem \ref{strictneg}.
\begin{proof}[Proof of Theorem \ref{strictneg}]
We follow the above outline. First we construct a two-ended manifold with untame ends. Then we construct a manifold with ends not quasi-isometric to a ray.

\begin{itemize}
\item[1)] \textbf{A two-ended manifold with untame ends:} Let $G$ be the graph whose vertices are elements of $\Z$ and edges are intervals $[m,m+1]$ for $m\in \Z$. Let $M$ the manifold obtained from step $1$ above. Let $Y^n$ be a compact hyperbolic manifold of dimension $n \geq 2$. Let $N^{n-1}$ be a compact, totally geodesic submanifold of $Y$. We require $N$ to be non-separating, i.e. $Y \setminus N$ is connected. Let $X$ be the block whose interior is diffeomorphic to $Y \setminus N$. We use the same kind of block $X$ for each vertex of $G$.  We can pick $X_0$ to correspond to the vertex $0 \in \Z$. The gluing maps are the identity map on each boundary component of each block.

\item[2)] \textbf{A manifold with ends not quasi-isometric to a ray:} Let $G$ be the graph whose vertices are elements of $\Z$. Two vertices $u$ and $v$ are connected by an edge $e_{ab}$ if $|a-b| =1$ or $a = b \ne 0$. 

Again, let $Y^n$ be a compact hyperbolic manifold of dimension $n \geq 2$. Let $N^{n-1}$ be a compact, totally geodesic submanifold of $Y$. We require $N$ to have two components this time, one of which is non-separating, and the other is not. There exist such pairs $(Y,N)$ (\cite{Millson},\cite{Lubotzky}). Let $X$ the block whose interior is diffeomorphic to the component $Y \setminus N$ that has three ends. So $X$ has three boundary components $A$, $B$, and $C$. Suppose that $A$ and $B$ correspond to the non-separating component of $N$.   

We use the same kind of block $X$ for each vertex of $G$. We can pick $X_0$ to correspond to the vertex $0 \in \Z$. The boundary components $A$ and $B$ correspond to the edges of the form $e_{m(m+1)}$ for $m \in \Z$, and the boundary component $C$ corresponds to the edges of the form $e_{m(-m)}$ for $m \in \Z$.

We follow the above outline to obtain a manifold $M$ with a singular metric $g$ but this time we choose the truncation of the cusps of type $C$ such that the following holds. Fix a point $x\in X$. For a given metric $g_X$ of $X$, let $l(A)$ be the distance from $A$ to $x$. One can think of $l(A)$ as the ``length" of the cusp $A$. We define $l(B)$ and $l(C)$ similarly. We choose the truncation of the cusps of type $C$ in such a way that for a given block $X$ corresponding to some vertex $m \in Z$, we have $ l(A)+l(B) <l(C) < 2(l(A)+l(B))$. This is not hard to obtain. 

Smooth out the metric $g$ to get a metric $h$. It is not hard to see that $(M,h)$ has two ends, one corresponds to the cusp $C$ of $X_0$ and the other is untame. The untame end of $M$ is not quasi-isometric to a ray.    
\end{itemize}


\end{proof}

\subsection{Slow curvature growth: proof of Theorem \ref{growth}} 
Let $f$ be a function as in Theorem \ref{growth}. We use the construction described in the first part of Theorem \ref{strictneg} to obtain a two-ended, finite-volume, negatively curved manifold $M$. We can pick $X_0$ to correspond to the vertex $0 \in \Z$. Let $X_i$ be the block corresponding to the vertex $i \in \Z$. Fix $p \in X_0$. We need to make sure that for each $r >0$, in a ball $B_p(r)$ of radius $r$ centered at $p$ we have $K < f(r)$ for large $r$. We just need to realize that all we need to do is for each $i >0$ (respectively, $i <0$), we need to truncate the cusp of $X_i$ far enough, where ``far enough" depends on the function $f$, down the cusp before gluing $X_{i+1}$ (respectively, $X_{i-1}$).

\section{Finite volume manifolds with infinitely many ends}\label{sec: infinite ends}
Now we give the proof of Theorem \ref{inftyend}.

\begin{proof}[Proof of Theorem \ref{inftyend}]
Let $G$ be a $3$-valent graph. Let $X$ be the block obtained in the second part of the proof of Theorem \ref{strictneg}. Apply the procedure in Section \ref{nonfinitetop}, we get the desired manifold. 
\end{proof}

We remark that there are examples of finite volume manifolds with curvature $-1\leq K\leq 0$ that have infinitely many ends. 
 
\bigskip
\noindent
\textbf{Examples with $-1\leq K(M) \leq 0$ with infinitely many ends.} This example is motivated by an example in \cite{Gromovneg}. Let $X$ be the manifold with boundary $\Sigma^4\times\mathbb{S}^1$, where $\Sigma^4$ is a surface with $4$ boundary components. Each boundary component of $X$ is a torus $\mathbb{S}^1_1\times \mathbb{S}^1_2$, one factor of which is a boundary component of $\Sigma^4$. The other factor corresponds to the $\mathbb{S}^1$ in $\Sigma^4\times\mathbb{S}^1$. 

Glue copies of $X$ together to get a noncompact manifold without boundary. Pairs of boundary components are identified with a flip, i.e. a swap of the $\mathbb{S}^1$ factors. The manifold $M$ obtained has the structure of a graph $G$ of spaces, where $G$ is the Cayley graph of $F_2 = \langle a_1, a_2\rangle$. We do the same trick as in \cite{Gromovneg} to obtain a finite volume manifold with curvature $-1 \leq K \leq 0$, which goes as follows.

It is not hard to show that for each $\varepsilon > 0$, there exists a hyperbolic metric $g_\varepsilon$ on $\Sigma^4$ such that a neighborhood of each boundary component of $\Sigma^4$ is isometric to the product of an interval with a circle of radius $\varepsilon$, and the volume $\Vol (\Sigma^4, g_\varepsilon) < 100$, and the diameter of $(\Sigma^4, g_\varepsilon)$ is greater than $1/10$. 

Let $X_0$ be a based block, corresponding to some based vertex $v_0 \in G$. We are going to give each block of $M$ a metric starting from $X_0$ and moving radially out. Give $X_0$ a metric $g_{X_0}$ that is isometric to the product of some hyperbolic metric on $\Sigma^4$ and $\mathbb{S}^1$ such that the diameter of $\mathbb{S}^1$ is no greater than the diameter of each boundary component of $\Sigma^4$. 

Suppose that we have put a metric $g_v$ on a block $X_v$ corresponding to a vertex $v$ that is a distance $k$ from $v_0$. Let $X_u$ be an adjacent block to $X_v$ that corresponds to a vertex $u$ that is a distance $k+1$ from $v_0$. We give $X_u$ a metric $g_u$ that is isometric to a product metric on $\Sigma^4\times \mathbb{S}^1$ (as above) such that the diameter of $\mathbb{S}^1$ is equal to the diameter of the boundary component of the $\Sigma^4$ factor of $X_v$ that is glued to $X_u$ and vice versa. We also requires $g_u$ to be such that the diameter of all the other boundary components of the $\Sigma^4$ factor of $X_u$ is less than $10^{-k-1}$. 

It is not hard to see that manifold $M$ with this metric is complete, has infinitely many ends, finite volume and nonpositive curvature bounded from below.

\bigskip
\noindent

\section{Curvature growth, volume decay, and topological finiteness}\label{sec:CGVD}

 
For a fixed dimension $n$, for each $b > 0$, let 
\[\mu_b = \dfrac{1}{b}\mu_1,\]
where $\mu_1$ is the Margulis constant. (Note that $\mu_b$ is the Margulis constant for the case where $0 > K > -b^2$.) 

It turns out that the Margulis constants depend only on local data. That is, we have the following local version of the Margulis lemma. 
\begin{lemma}[Local Little Loop Lemma]\label{LLLL}
Let $\widetilde{M}$ be a complete, simply connected, negatively curved manifold. Let $p \in \widetilde{M}$, and let $a > 0$. If the sectional curvature of $\widetilde{M}$ in the ball $B_p(10\mu_a)$ centered at $p$ with radius $10\mu_a$ is between $0$ and $-a^2$, then any discrete group of isometries of $\widetilde{M}$ generated by elements that move $p$ a distance less than $\mu_a$ is virtually nilpotent.
\end{lemma}

\begin{proof}
The proof of this is exactly the proof of the Margulis lemma (e.g. see \cite{BuserKarcher}). One needs to check that every step works because what is needed is a control on parallel transport in a large enough region of $\widetilde{M}$, which will follow if the curvature in the region is sufficiently bounded as in the hypothesis of Lemma \ref{LLLL}. 
\end{proof}

For $q \in M$, let $\Gamma_q$ be the subgroup generated by loops based at $q$ with length less than $\mu_{b(q)} := \mu_{b(d(p,q))}$. Then $\Gamma_q$ is virtually nilpotent by Lemma \ref{LLLL}.  

The proof of Theorem \ref{CGVD} is similar to that in \cite{BGS} and will be given at the end of this section. The difference is one needs to keep track of local data. 
\begin{lemma}\label{no small hyperbolic}
There is $R > 0$ such that if $q \in M \setminus B_p(R)$, then $\Gamma_q$ is nontrivial and contains only parabolic isometries.
\end{lemma}

\begin{proof}
Since
\[\lim_{r\rightarrow \infty} b(r)^n\Vol(B_p(r) \setminus B_p(r - 1))^2 = 0.\]
there is $R>0$ such that for each $q \in M \setminus B_p(R)$, the injectivity radius at $q$ is less than $\mu_{b(q)}$, so $\Gamma_q$ is nontrivial. Pick $R$ large enough so that 

\[b(r)^n\Vol(B_p(r) \setminus B_p(r - 1))^2 < 10^{-100},\]
and
\[\Vol(B_p(r) \setminus B_p(r - 1))^2  < 10^{-100}.\]


Let $q \in M \setminus B_p(R)$. By Lemma \ref{LLLL}, the group $\Gamma_q$ has a finite index subgroup $\Lambda_q$ that is nilpotent. It follows that if $\Gamma_q$ contains a hyperbolic isometry, then $\Lambda_q \cong \Z$ and contains only hyperbolic isometries. But if  $\Gamma_q$ contains a parabolic isometry, then $\Lambda_q$ contains only parabolic isometries. 

It is not hard to see that for $\varepsilon > 0$ small enough, if $x \in B_q(\varepsilon)$, then $\Gamma_x$ contains the same type of isometries as $\Gamma_q$. Hence we can define a function 
\[f \colon M\setminus B_p(R) \longrightarrow \{\text{hyp},\text{para}\}\] 
defined as $f(x) = \text{hyp}$ (respectively, $f(x) = \text{para}$) if $\Gamma_x$ contains a hyperbolic (respectively, parabolic) isometry. By the previous paragraph, $f$ is locally constant. Hence, $f$ is constant on each connected component of $M \setminus B_p(R)$. 

We claim that for each component $C$ of $M \setminus B_p(R)$, we must have $f(C) = \text{para}$. Suppose this is not true for some $C$. Let $x\in C$. Then $\Gamma_x$ contains some hyperbolic isometry $h$ with axis a geodesic $\gamma(t)$. Then the fixed point set on $\partial_\infty \widetilde{M}$ is $\{\gamma(-\infty), \gamma(\infty)\}$. By a similar argument as above, the function that assigns to each point $x \in C$ the set of fixed points in $\partial_\infty\widetilde{M}$ of $\Gamma_x$ is also locally constant. Hence, for any $y\in C$ and $\Gamma_y = \Gamma_x$. But this is impossible if $y$ is a large enough distance from $x$ since the displacement function $d(z,\gamma(z))$ of a hyperbolic isometry $h$ with axis $\gamma$ goes to $\infty$ as $d(z,\gamma)\rightarrow\infty$.

Therefore, $f(C) = \text{para}$ for all components $C$ of  $M \setminus B_p(R)$, and the lemma follows.

\end{proof}
Now we define a Morse function on $M$ outside some compact set. 
For a discrete group $\Gamma$ of isometries of $\widetilde{M}$, we define the following displacement function on $\widetilde{M}$ as in \cite{Gromovneg}. 
\[\delta_\Gamma (x) = \min_{\gamma\in\Gamma \setminus\{1\}} d(x,\gamma (x)).\]
We see that $\delta_\Gamma$ is $\Gamma$-invariant and thus, descends to a function on $\widetilde{M}/\Gamma$. In general, the function $\delta_\Gamma$ is not smooth but locally is equal to the minimum of finitely many smooth functions $f_1, f_2, ..., f_k$, namely $f_i (x) = d(x, \gamma_i(x))$ for some deck transformation $\gamma_i \in \pi_1(M)$. A point $x\in M$ is not a critical point if there is a tangent vector $v \in T_xM$ such that $df_i (v) > 0$ for all $i = 1, 2,...,k$. 

\begin{lemma}\label{Morse function}
Let $R$ be as in Lemma \ref{no small hyperbolic}. Then the function $\delta_{\pi_1(M)}$ does not have any critical point in $M\setminus B_p(R)$. 
\end{lemma}

\begin{proof}
Let $q \in M \setminus B_p(R)$. Since $\delta_{\pi_1(M)}(q)$ is realized by a parabolic isometry $\gamma$, there is a direction of increasing for $\delta_{\pi_1(M)}$ at $q$ (which is pointing the fixed point at infinity of $\gamma$). So $\delta_{\pi_1(M)}$ has no critical points outside $B_p(R)$.
\end{proof}
Now we prove Theorem \ref{CGVD}.
\begin{proof}[Proof of Theorem \ref{CGVD}]
This follows from Lemma \ref{Morse function} and Morse theory. The function $\delta_{\pi_1(M)}$ is not smooth, but Morse theory applies to functions that locally are the minimum of finitely many functions with no critical points as described above. (See also \cite[p. 226]{Gromovneg}.) 

Let $\alpha < \delta_\Gamma(B_p(R))$. Then $\alpha$ is a regular value of $\delta_\Gamma$. There is $R_1 > R$ such that $\delta_\Gamma(M \setminus B_p(R_1)) \leq \beta < \alpha$. The level set of $\delta_\Gamma = \alpha$ is an embedded submanifold of $M$ that is contained in the bounded ball $B_p(R_1)$. Thus, $\delta_\Gamma^{-1}(\alpha)$ is compact. By Morse theory, $M$ is diffeomorphic to the interior of a compact manifold $\overline{M}$ with boundary diffeomorphic to $\delta_\Gamma^{-1}(\alpha)$.
\end{proof}


\section{Negatively curved, invisibility manifolds}\label{sec: invisibility}
\subsection{An invisibility, negatively curved surface} 
We construct a complete, negatively curved metric on $\R^2$ such that the resulting Riemannian manifold fails to satisfy the Visibility axiom. (Compare with \cite[p. 53]{BGS}).

Let $g \colon \R \longrightarrow\R$ be a smooth function that satisfies 
\[ g''(t) > 0, \quad \text{and} \quad g'(\infty) - g'(-\infty) = \dfrac{\pi}{10}.\] 
Let $f\colon \R^2 \longrightarrow\R$ be defined as
\[f(x,y) = g(x) - g(y).\]
Let $\widetilde{M}$ be the graph of $f$ with the induced Riemannian metric from $\R^3$, and let $(x,y)$ be a local parametrization of $\widetilde{M}$. The Gaussian curvature at $(x,y)$ is 
\[\kappa(x,y) = \dfrac{f_{xx}f_{yy} -f_{xy}^2}{(1 + f_x^2 + f_y^2)^2} 
 = \dfrac{-g''(x)g''(y)}{(1+g'(x)^2 + g'(y)^2)^2} < 0. \]
We have the integral of the Gaussian curvature over $\widetilde{M}$ satisfies the following bounds.
\begin{align*}
0 \geq \int_{\widetilde{M}}\kappa(x,y) \geq & -\int_{-\infty}^{\infty}g''(x)dx\times\int_{-\infty}^{\infty}g''(y)dy \\
& = -(g'(\infty) - g'(-\infty))^2 \\
& = \left(\dfrac{\pi}{10}\right)^2 = -\dfrac{\pi^2}{100}.
\end{align*}
\begin{theorem}
The above $\widetilde{M}$ with the induced Riemannian metric from $\R^3$ is a negatively curved, invisibility manifold. 
\end{theorem}
\begin{proof}
Suppose that $\widetilde{M}$ is a visibility manifold. Let $p \in M$ and let $x, y \in\partial_{\infty}\widetilde{M}$ such that the angle $\alpha$ between the geodesic rays $px$ and $py$ is $\dfrac{\pi}{100}$. By assumption, there is a geodesic $\lambda$ connecting $x$ and $y$. Let $\triangle$ be the triangle with vertices $p,x,y$. Then by the Gauss-Bonnet theorem, we have
\[\int_{\triangle}\kappa = \alpha - \pi = \dfrac{\pi}{100} -\pi. \]
But 
\[ \int_{\triangle}\kappa \geq \int_{\widetilde{M}}\kappa \geq  -\dfrac{\pi^2}{100}, \]
which is a contradiction. Hence, there is no geodesic connecting $x$ and $y$, i.e. $\widetilde{M}$ is not a visibility manifold.
\end{proof}

\subsection{Proof of Theorem \ref{nonvisibility lattice}}

We will prove Theorem \ref{nonvisibility lattice} for the case where $M$ has dimension $n = 3$ since the proof involves some writing things out in coordinates. The general case can be reduced to this case.

Let $N$ be a compact hyperbolic manifold of dimension $n$. Let $H$ be a connected, totally geodesic, embedded, codimension $2$ submanifold of $N$. Such pairs $N$ and $H$ can be constructed arithmetically. Since we assume that $n = 3$, the submanifold $H$ is a geodesic loop. Let $h$ be the length of the geodesic loop $H$. The hyperbolic metric of $N$ on a tubular neighborhood $T_\varepsilon H$ of $H$ is 
\[ g_{hyp} = \cosh^2(r) du^2 + \sinh^2(r) d\theta^2 + dr^2.\] 
Let $M = N\setminus H$. Let $g$ be the metric constructed in \cite{Fujiwara}. We have
\[ g =  \cosh^2(r) du^2 + \sinh^2(r) d\theta^2 + f(r)^2dr^2,\] 
where $f(r)$ is convex, $f(r) \rightarrow \infty$ as $r \rightarrow 0$, and some other properties (in \cite{Fujiwara}).

The submanifold $\theta = \text{const}$ of $T_\varepsilon H$ is totally geodesic since the metric $g$ is invariant under reflection across $\theta = \text{const}$ fixing $r$ and $u$. Let $S$ be such a submanifold $\theta = a$ for some constant $a$. It is not hard to show that $S$ is isometric to a surface of revolution obtained by rotating the graph of a function $x = \varphi(z)$ around the $z$-axis in $\R^3$. We can pick $\varphi(z)$ to be defined on $(1,\infty)$. The function $\varphi (z)$ must 
satisfy $\varphi(z) \rightarrow h$ as $z \rightarrow \infty$. 

We identify $S$ with the surface of revolution generated by $\varphi$. Let $\gamma (t)$, for $t \in [0,\infty)$ be a geodesic ray on $S$. 
For each $q \in S$, let $\rho(q)$ be the Euclidean distance between $q$ and the $z$-axis.  For each $t$, let $\alpha (\gamma(t))$ be the angle between the tangent vector $\gamma' (t)$ and the meridian of $S$. By the conservation of angular momentum we have the quantity $\rho (\gamma(t)).\sin(\alpha (\gamma(t))) = \text{const}$ along $\gamma(t)$ (see \cite[p. 46]{Arnold}).

Fix $p \in S$. Let $\rho_0$ be the Euclidean distance between $p$ and the $z$-axis. Let $\alpha_0 > 0$ be such that $\rho_0\sin(\alpha_0) < h$. Let $\gamma_1$ (respectively $\gamma_2$) be the geodesic ray starting at $p$ with initial tangent vector $\gamma'(0)$ at angle $\alpha_0$ (respectively $-\alpha_0$) with the meridian. 

Consider $\gamma_1$. For each $t \in [0, \infty)$, we have 
\[\rho(\gamma_1(t))\cdot\sin(\alpha(\gamma_1(t))) = \rho(\gamma_1(0))\cdot\sin(\alpha(\gamma_1(0))) = \rho_0\cdot\sin(\alpha_0).\] 
Since $\rho_0\sin(\alpha_0) < h$ and $\rho(\gamma_1(t)) > h$ for all $t \in [0,\infty)$, it follows that $\alpha(\gamma_1(t)) < \pi/2 - \epsilon$ for some $\epsilon>0$. Therefore, $\gamma_1[0,\infty)$ is contained in $S$. Similarly, $\gamma_2[0,\infty)$ is contained in $S$.  

Let $\widetilde{p}$ be a lift of $p$ in the universal cover $\widetilde{M}$. Let $\widetilde{\gamma}_i$ (for $i = 1, 2$) be lifts of $\gamma_i$ starting at $\widetilde{p}$. Then $\widetilde{\gamma}_1(\infty)$ and $\widetilde{\gamma}_2(\infty)$ are distinct points on $\partial_\infty \widetilde{M}$ since the angle $\angle_{\widetilde{p}}(\widetilde{\gamma}_2'(0), \widetilde{\gamma}_1'(0)) = 2\alpha_0 \ne 0$. We claim that $\widetilde{\gamma}_1(\infty)$ and $\widetilde{\gamma}_2(\infty)$ cannot be connected by a geodesic in $\widetilde{M}$. 

Suppose that there is a geodesic $\widetilde{\lambda} \in \widetilde{M}$ connecting $\widetilde{\gamma}_1(\infty)$ and $\widetilde{\gamma}_2(\infty)$. Then given sequences $a_n \rightarrow \infty$ and $b_n \rightarrow \infty$, the geodesic segments $\widetilde{\lambda}_n$ connecting $\widetilde{\gamma}_1(a_n)$ and $\widetilde{\gamma}_2(b_n)$ must converge to $\widetilde{\lambda}$. We will show that this is not possible by showing that such $\widetilde{\lambda}_n$ must leave all compact sets in $\widetilde{M}$.   

Let $\lambda_n$ be the image of $\widetilde{\lambda}_n$ under the covering space projection. Since $\gamma_1(a_n) \in S$, $\gamma_2(b_n)$ and $S$ is totally geodesic in $M$, it follows that $\lambda_n \subset S$. As shown above, for $i = 1,2$, we have $|\alpha(\gamma_i(t))| < \pi/2 - \epsilon$ for some $\epsilon >0$ , so the $z$-coordinate $z(\gamma_i(t)) \rightarrow \infty$ as $t\rightarrow\infty$. Therefore $z(\gamma_1(a_n)) \rightarrow \infty$ and $z(\gamma_2(b_n)) \rightarrow \infty$. Since width of $S$ is strictly decreasing with respect to $z$, it follows that for any point $q\in\lambda_n$, we have $z(q) \geq \min\{\gamma_1(a_n),\gamma_2(b_n)\}$. Thus $\lambda_n$ leaves every compact sets of $M$. Therefore no such $\widetilde{\lambda}$ exists. Hence $\widetilde{M}$ does not satisfy the visibility axiom and we are done with proving Theorem \ref{nonvisibility lattice} for the case $n = 3$.

Now we address the case where $M$ has dimension $n >3$. As before, $M = N\setminus H$. Let $W$ be a geodesic loop in $H$. Let $V$ be an normal $\varepsilon$-neighborhood of $W$. Then $V$ is a totally geodesic $3$-dimensional submanifold in $M$. The above argument applies to $V$ and gives that $\widetilde{M}$ is not a visibility manifold. So we are done with proving Theorem \ref{nonvisibility lattice}.
\bibliographystyle{amsplain}
\bibliography{bibliography}

\end{document}